\documentclass[10pt]{amsart}
\usepackage{amssymb}

\usepackage{blindtext}
\usepackage[utf8]{inputenc}
\newtheorem{thm}{Theorem}[section]
\newtheorem{ques}{Question}

\theoremstyle{definition}
\newtheorem{cor}[thm]{Corollary}
\newtheorem{lem}[thm]{Lemma}
\newtheorem{prop}[thm]{Proposition}
\newtheorem{defn}[thm]{Definition}

\theoremstyle{remark}

\numberwithin{equation}{section}

\begin{document}

\title[Generic behavior of a measure preserving transformation]{Generic behavior of a measure preserving transformation}

\author{Mahmood Etedadialiabadi}

%\thanks{Research supported by NSF grant DMS-1266189.}

\address{Department of Mathematics\\
University of Illinois\\
1409 W. Green St.\\
Urbana, IL 61801, USA}

\email{etedadi2@math.uiuc.edu}

%\subjclass[2000]{22A25, 46A16, 46E10, 46E30}


\begin{abstract}
Del Junco--Lema\'nczyk [\ref{DL_B}] showed that a generic measure preserving transformation satisfies a certain orthogonality conditions. More precisely, there is a dense $G_\delta$ subset of measure preserving transformations such that for every $T\in G$ and $k(1), k(2), \dots, k(l)\in \mathbb{Z}^+$, $k'(1), k'(2), \dots, k'(l')\in \mathbb{Z}^+$, the convolutions
\[
\sigma_{T^{k(1)}}\ast\cdots\ast\sigma_{T^{k(l)}} \ \text{and} \ \sigma_{T^{k'(1)}}\ast\cdots\ast\sigma_{T^{k'(l')}}
\]
where $\sigma_{T^k}$ is the maximal spectral type of $T^k$, are mutually singular, provided that $(k(1), k(2), \dots, k(l))$ is not a rearrangement of $(k'(1), k'(2), \dots, k'(l'))$. We will introduce analogous orthogonality conditions for continuous unitary representations of the group of all measurable functions with values in the circle, $L^0(\mu,\mathbb{T})$, which we denote by the DL--condition. We connect the DL--condition with a result of Solecki [\ref{S1_B}] which identifies continuous unitary representations of $L^0(\mu,\mathbb{T})$ with a collection of measures $\{\lambda_\kappa\}$ where $\kappa$ runs over all increasing finite sequence of non-zero integers. In particular, we show that the ``probabilistic'' DL--condition translates to ``deterministic'' orthogonality conditions on the measures $\lambda_\kappa$. As a corollary, we show that the same orthogonality conditions as in the result by Del Junco--Lema\'nczyk hold for a generic unitary operator on a separable infinite dimensional Hilbert space.%Also, we show a certain notion of disjointness for generic functions in $L^0(\mu,\mathbb{T})$.
\end{abstract}

\maketitle
\bigskip\noindent

\section{Introduction}
%We call a homomorphism $\Pi:\Gamma\rightarrow G$ a \textbf{representation} of $\Gamma$ in $G$ where $\Gamma$ is a countable group and $G$ is a Polish group. Recall that a topological group $G$ is Polish if its topology is separable and is induced by a complete metric. %We study generic behavior of a representation of $\Gamma=\mathbb{Z}$ in $G=\operatorname{Aut}(X,\mu)$ where $(X,\mu)$ is a Borel measure space and $\mu$ is non-atomic. 
%where $(X,\mu)$ is a Borel Probability measure space, $\operatorname{Aut}(X,\mu)$ is the set of measure preserving transformations on $(X,\mu)$, and $\mathcal{U}(H)$ is the group of unitary operators on a Hilbert space, $H$, with the strong topology.
%More precisely, we are interested in the behavior of $\overline{\langle T\rangle}$, the closure of the group generated by $T$, for a generic $T\in \operatorname{Aut}(X,\mu)$. In particular, does there exist a topological group, $G$, such that for a generic $T\in \operatorname{Aut}(X,\mu)$, $\overline{\langle T\rangle}$ is isomorphic to $G$. 
The main spectral property of a generic $T\in\operatorname{Aut}(X,\mu)$ is given by Del Junco--Lema\'nczyk [\ref{DL_B}] where $\operatorname{Aut}(X,\mu)$ is the set of all measure preserving transformations on a Borel measure space $(X,\mu)$. They proved that for a generic $T\in\operatorname{Aut}(X,\mu)$, for every $k(1), k(2), \dots, k(l)\in\mathbb{Z}^+$ and $k'(1), k'(2), \dots, k'(l')\in \mathbb{Z}^+$, the convolutions
	\[
	\sigma_{T^{k(1)}}\ast\cdots\ast\sigma_{T^{k(l)}} \ \text{and} \ \sigma_{T^{k'(1)}}\ast\cdots\ast\sigma_{T^{k'(l')}}
	\]
where $\sigma_{T^k}$ is the maximal spectral type of $T^k$, are mutually singular, provided that $(k(1), \dots, k(l))$ is not a rearrangement of $(k'(1), \dots, k'(l'))$. 
Earlier, Katok [\ref{Kat_B}], Stepin [\ref{Ste_B}], Choksi and Nadkarni [\ref{Cho_B}] proved special cases of the above property when $l=l'=1$ and when all the powers $k(i),k'(j)$, are equal to $1$. 

By analogy with the above Del Junco--Lema\'nczyk [\ref{DL_B}] statement, we define an orthogonality condition for a continuous unitary representation of $L^0(\mu,\mathbb{T})$ (we justify the analogy later in the introduction). Here $L^0(\mu,\mathbb{T})$ is the set of all $\mu$-equivalence classes of Borel measurable functions from $(X,\mu)$ to the unit circle. We consider $L^0(\mu,\mathbb{T})$ with pointwise multiplication and convergence in measure topology which makes it a Polish group. We say a continuous unitary representation of $L^0(\mu,\mathbb{T})$,
\[
\Phi:L^0(\mu,\mathbb{T})\rightarrow \mathcal{U}(H)
\]
satisfies the \textbf{DL--condition} if for a generic Borel measurable function $f\in L^0(\mu,\mathbb{T})$, for every ${k(1), \dots, k(l)\in \mathbb{Z}^+}$, $k'(1), \dots, k'(l')\in \mathbb{Z}^+$, the convolutions
	\[
	\sigma_{\Phi(f)^{k(1)}}\ast\cdots\ast\sigma_{\Phi(f)^{k(l)}} \ \text{and} \ \sigma_{\Phi(f)^{k'(1)}}\ast\cdots\ast\sigma_{\Phi(f)^{k'(l')}}
	\]
	are mutually singular, provided that $(k(1), \dots, k(l))$ is not a rearrangement of $(k'(1), \dots, k'(l'))$. Here $\mathcal{U}(H)$ is the group of all unitary operators on a separable infinite dimensional Hilbert space $H$. We consider $\mathcal{U}(H)$ with the strong topology, that is, pointwise convergence. Note that $\mathcal{U}(H)$ with the strong topology is a Polish group.
		
This raises the natural question of characterizing unitary representations of $L^0(\mu,\mathbb{T})$ that satisfy the DL--condition. Our aim is to give such a characterization in terms of the objects appearing in the following result of Solecki [\ref{S1_B}]. Solecki [\ref{S1_B}, Theorem $2.1$] showed that every continuous unitary representation of $L^0(\mu,\mathbb{T})$ is a direct sum of action by pointwise multiplication on measure spaces $(X^{|\kappa|},\lambda_\kappa)$ where $\kappa$ is an increasing finite sequence of non-zero integers and $\lambda_\kappa$ is a finite measure on $X^{|\kappa|}$ whose marginal measures are absolutely continuous with respect to $\mu$.

 In this paper, we characterize unitary representations of $L^0(\mu,\mathbb{T})$ that satisfy the DL--condition in terms of orthogonality conditions on the measures $\lambda_\kappa$. We show that a unitary representation of $L^0(\mu,\mathbb{T})$, $\Phi=\bigoplus\sigma(\kappa,\lambda_\kappa)$, satisfies the DL--condition if and only if we have (see Definition \ref{r} and Definition \ref{k})
\begin{equation}\label{equivalence}
\lambda_{\kappa_1}\times\lambda_{\kappa_2}\times\cdots\times\lambda_{\kappa_l} \perp r(\lambda_{\kappa_{1}'}\times\lambda_{\kappa_{2}'}\times\dots\times\lambda_{\kappa_{l'}'})
\end{equation}
for every $(\kappa_1, \kappa_2, \dots, \kappa_l)$, $(\kappa_1', \kappa_2', \dots, \kappa_{l'}')$, and $r\in \operatorname{S}_t$ where $\operatorname{S}_t$ is the permutation group of $\{1,2,\dots,t\}$, such that
\[
k(1)\kappa_1+k(2)\kappa_2+\cdots+k(l)\kappa_l= r\Big(k'(1)\kappa_{1}'+k'(2)\kappa_{2}'+\cdots+k'(l')\kappa_{l'}'\Big) 
\]
for some non-zero integers $(k(1), k(2), \dots, k(l))$ and $(k'(1), k'(2), \dots, k'(l'))$, provided that one is not a rearrangement of the other. Note that this equivalence translates the ``probabilistic'' statement of the DL--condition to a ``deterministic'' condition on the measures $\lambda_\kappa$.

Now we give justification for the above analogy. There is a question by Glasner and Weiss which asks is it true that for a generic $T\in \operatorname{Aut}(X,\mu)$
\[
\overline{\langle T\rangle}\cong L^0(\mu,\mathbb{T})?
\]
In fact, they seem to hint that the answer is positive.  %In particular, for a generic measure preserving transformation, $T$, $\overline{\langle T\rangle}$ is L$\acute{e}$vy (note that $L^0(\lambda,\mathbb{T})$ is L$\acute{e}$vy). 

There is evidence to support the suggestion by Glasner and Weiss. In the following, we mention some of the known results about generic behavior of $T\in \operatorname{Aut}(X,\mu)$ and state how they support the suggestion by Glasner and Weiss. Mellaray--Tsankov [\ref{MT_B}, Theorem $1.4$] proved that for a generic measure preserving transformation $T$, $\overline{\langle T\rangle}$ is extremely amenable. Recall that a Polish group is called \textbf{extremely amenable} if every continuous action of the group on a compact Hausdorff space has a fixed point. Note that $L^0(\mu,\mathbb{T})$ is extremely amenable. 

%Furthermore, Chacon--Schwartzbauer [\ref{Cha_B}] showed that for a generic measure preserving transformation, $T$, $C(T)=\overline{\langle T\rangle}$ where
%\[
%C(T)=\{S\in \operatorname{Aut}(X,\mu) \ : \ ST=TS\}.
%\] 
%This was originally proved by  and another proof was presented by King [\ref{Kin2_B}] who showed that the conclusion holds for a rank-$1$ transformation (note that a generic measure preserving transformation has rank-$1$). 
Furthermore, Glasner--Weiss [\ref{GW_B}, Theorem $5.2$] proved that for a generic measure preserving transformation, $T\in \operatorname{Aut}(X,\mu)$, the action of $\overline{\langle T\rangle}$ on the Borel measure space $(X,\mu)$ is whirly. Recall that an action of a topological group $G$ on a Borel measure space $(X,\mu)$ is whirly if for every Borel subset of $X$ with a positive measure, $A$, and every non-empty open neighborhood of the identity of $G$, $U$, $\mu(UA)=1$. Note that ergodic actions of $L^0(\mu,\mathbb{T})$ are whirly [\ref{GW_B}, Theorem $3.11$].

   Moreover, Solecki [\ref{S2_B}, Corollary $2$] showed that for a generic measure preserving transformation, $T$, $\overline{\langle T\rangle}$ is a continuous homomorphic image of a closed linear subspace of $L^0(\lambda,\mathbb{R})$, where $\lambda$ is the Lebesgue measure on the set of real numbers. Obviously, $L^0(\lambda,\mathbb{T})$ is a continuous homomorphic image of $L^0(\lambda,\mathbb{R})$ with $f\to e^{if}$.
	
	Finally, Mellaray--Tsankov [\ref{MT_B}, Theorem $4.4$] considered generic behavior of a unitary operator on a separable infinite dimensional Hilbert space. They showed that in $\mathcal{U}(H)$, the unitary group of a separable infinite dimensional Hilbert space $H$ with the strong topology, for a generic $u\in\mathcal{U}(H)$
\[
\overline{\langle u\rangle}\cong L^0(\mu,\mathbb{T}).
\]
Note that $\operatorname{Aut}(X,\mu)$ is isomorphic to a closed subset of $\mathcal{U}(L^2(X,\mu))$.

Assuming that 
\begin{equation}\label{equationintro}
\overline{\langle T\rangle}\cong L^0(\mu,\mathbb{T})
\end{equation}
for a measure preserving transformation $T\in \operatorname{Aut}(X,\mu)$, we can define a continuous unitary representation of $L^0(\mu,\mathbb{T})$ by 
\[
\Psi:L^0(\mu,\mathbb{T}) \cong \overline{\langle T\rangle} \hookrightarrow \operatorname{Aut}(X,\mu)\subseteq \mathcal{U}(L^2(X,\mu)).% \ \ \ \text{where} \ \ \ U_T(f)=f\circ T^{-1}.
\]
%This raises a natural question about the generic behavior of $\overline{\langle T\rangle}$. The following question was raised by Mellaray--Tsankov [\ref{MT_B}], Solecki, and Pestov: 
%\begin{ques}\label{q1}
%Does there exist a Polish group $G$ such that for a generic measure preserving transformation, $T$, $\overline{\langle T\rangle}\cong G$? In particular, is this true for $G=L^0(\mathbb{T})$?
%\end{ques}
%Also, Glasner--Weiss [\ref{GW_B}] asked the following question: 
%\begin{ques}\label{q2}
%is $\overline{\langle T\rangle}$ L$\acute{e}$vy for a generic measure preserving transformation $T$? 
%\end{ques}
This suggests that properties of measure preserving transformations with (\ref{equationintro}) can be translated to properties of continuous unitary representations of $L^0(\mu,\mathbb{T})$.

%Assuming the answer is positive
Furthermore, if a generic $T\in \operatorname{Aut}(X,\mu)$ satisfies (\ref{equationintro}), then by Solecki [\ref{S2_B}, Lemma 3]), a generic $T\in \operatorname{Aut}(X,\mu)$ is in the range of a continuous unitary representation of $L^0(\mu,\mathbb{T})$ such that the representation satisfies the DL--condition. This property suggests an approach to answer the question by Glasner and Weiss. In particular, if one could show that a generic $T\in \operatorname{Aut}(X,\mu)$ is in the range of a continuous unitary representation of $L^0(\mu,\mathbb{T})$ with the DL--condition, then this strengthen the suggestion that the answer to the question is positive. On the contrary, if one could show that for a non-meager $T\in \operatorname{Aut}(X,\mu)$, $T$ is not in the range of a continuous unitary representation of $L^0(\mu,\mathbb{T})$ with the DL--condition, then the answer to the question is negative.

  In Section $3$, we define a notion of disjointness for a $f\in L^0(\mu,\mathbb{T})$ and we show that a generic $f\in L^0(\mu,\mathbb{T})$ satisfies this notion of disjointness. In Section $4$ and $5$, we use the result of Section $3$ to prove the equivalence (\ref{equivalence}). In Section $6$, we use the equivalence (\ref{equivalence}) to show that the orthogonality conditions proved by Del Junco--Lema\'nczyk [\ref{DL_B}] also hold for a generic $u\in \mathcal{U}(H)$, that is, the convolutions
	\[
	\sigma_{u^{k(1)}}\ast\cdots\ast\sigma_{u^{k(l)}} \ \text{and} \ \sigma_{u^{k'(1)}}\ast\cdots\ast\sigma_{u^{k'(l')}}
	\]
	are mutually singular, provided that $(k(1), k(2), \dots, k(l))$ is not a rearrangement of $(k'(1), k'(2), \dots, k'(l'))$.
 
\bigskip\noindent

\section{Preliminaries}
In this section, we introduce the notation that we use in the paper. By a standard Borel measure space $(X,\mu)$ we mean a standard Borel space, $X$, equipped with a non-atomic Borel probability measure on $X$, $\mu$. All such spaces are Borel isomorphic to $([0,1],\lambda)$, where $\lambda$ is the Lebesgue measure on the Borel subsets of $[0,1]$. We denote by $\operatorname{Aut}(X,\mu)$, and sometimes just $\operatorname{Aut}(\mu)$ when $X$ is understood, the group of Borel automorphisms of $X$ which preserve the measure $\mu$, that is for every $T\in \operatorname{Aut}(X,\mu)$ and $A$, a Borel subset of $X$, $A$, $T(A)$, and $T^{-1}(A)$ have the same measure. In which we identify two Borel automorphisms if they coincide $\mu$--almost everywhere, that is, they coincide on a full measure subset of $X$. When we talk about Borel subsets of $(X,\mu)$, we usually consider them up to null sets.\\

There are two fundamental topologies on $\operatorname{Aut}(X,\mu)$, the weak and the uniform topology. In this paper, we consider $\operatorname{Aut}(X,\mu)$ with the weak topology. Denote by $\text{MALG}_\mu$ the measure algebra of $\mu$, that is, the algebra of
Borel subsets of X modulo null sets. It is a Polish Boolean algebra under the topology given by the complete metric
\[
d(A, B) = d_\mu(A, B) = \mu(A\Delta B),
\]
where $\Delta$ is the symmetric difference. The weak topology is the topology generated by the functions 
\[
T\mapsto T(A), \ A\in \text{MALG}_\mu.
\]
With the weak topology, which we denote by $w$, $(\operatorname{Aut}(X,\mu),w)$ is a Polish topological group. A compatible left-invariant metric is given by
\[
\delta_w(S,T)=\sum 2^{-n} \mu(S(A_n)\Delta T(A_n)),
\]
where ${A_n}$ is a dense subset of $\text{MALG}_\mu$. We encourage the reader to view [\ref{K2_B}, Chapter I] for more information on the topological group $\operatorname{Aut}(X,\mu)$.\\

Let $(X,\mu)$ be a standard Borel measure space and $G$ be a topological group. By $L^0(X,\mu, G)$, or just $L^0(\mu, G)$ when $X$ is understood, we denote the topological group of all $\mu$-equivalence classes of Borel measurable functions on $X$ with values in $G$. We consider $L^0(\mu, G)$ with the pointwise multiplication and the convergence in measure topology. Furthermore, we assume that the measure $\mu$ is non-atomic. For a tuple $\kappa=(k_1,k_2,\dots,k_m)$ in $\mathbb{Z}$ and a function $f\in L^0(\mu, G)$, we define the function $f^\kappa: X^m\rightarrow G$ by
\begin{equation}\label{function}
f^\kappa(x_1,x_2,\dots,x_m)=\prod_{i=1}^m (f(x_i))^{k_i}.
\end{equation}
Moreover, for $A_1,\dots,A_k\subseteq X$ and $f_1,\dots,f_k\in L^0(\mu,G)$, we define
\[
f_1(A_1)\cdots f_k(A_k):=\{f_1(x_1)\cdots f_k(x_k) \ : \ \text{for every } x_1\in A_1,\dots, x_k\in A_k\}.
\]
In this paper, we usually consider $L^0(\mu, \mathbb{T})$, where $\mathbb{T}$ is the unit circle (also denoted by $S^1$ in the literature). Note that the set of all open subsets of the form
\[
U_{\epsilon,f_0}=\{f\in L^0(\mu,\mathbb{T}) \ : \ \int{\left|f-f_0\right|} d\mu < \epsilon \}
\]
where $\epsilon>0$ is a real number and $f_0\in L^0(\mu,\mathbb{T})$, is a basis for the topology on $L^0(\mu,\mathbb{T})$.
 
We denote the set of Borel probability measures on a standard Borel space, $X$, by $\mathbb{P}(X)$. %The topology on $\mathbb{P}(X)$ is generated by the following basic open sets: let $\Phi:X\rightarrow \mathbb{R}$ be a bounded and continuous function, $\epsilon>0$, and $r$ be a real number. Then,
%\[
%U=\{\mu \in\mathbb{P}(X) \ : \ \left|\int_\mathbb{T}{\Phi(z)} d\mu-r\right|<\epsilon\}
%\]
%is a basic open set. $\mathbb{P}(X)$ with this topology is a Polish space.
Furthermore, $S_n$ denotes the permutation group of $\{1,2,\dots,n\}$. 

We introduce the following definitions which will be used to state the main theorem of Section $4$ (Theorem \ref{MT}).
\begin{defn}\label{r}
Let $r\in \operatorname{S}_n$, then $r$ induces three functions, all of which denoted by the same operator name, $r$:
\begin{enumerate}
	\item a function $r: X^n\rightarrow X^n$ which sends $(x_1,\dots,x_n)$ to $(x_{r(1)},\dots,x_{r(n)})$,
	
	\item a function $r: \mathbb{P}(X^n)\rightarrow \mathbb{P}(X^n)$ which sends $\mu\in\mathbb{P}(X)$ to the measure obtained by permuting coordinates of $X^n$ with respect to $r$,
	
	\item a function $r: \mathbb{Z}^n\rightarrow \mathbb{Z}^n$ which sends the tuple $\Big(k(1), k(2), \dots, k(n)\Big)$ to ${\bigg(k\Big(r(1)\Big), k\Big(r(2)\Big), \dots, k\Big(r(n)\Big)\bigg)}$.
\end{enumerate}
\end{defn}

\begin{defn}\label{k}
Let $\kappa=(k(1),k(2),\dots,k(l)),\kappa'=(k'(1),k'(2),\dots,k'(l'))$ be two sequences of integer numbers. We define:
\begin{enumerate}
\item $\kappa+\kappa'=(k(1),k(2),\dots,k(l),k'(1),k'(2),\dots,k'(l'))$,
\item for $d\in \mathbb{Z}$, $d\kappa=(dk(1),dk(2),\dots,dk(l))$.
\end{enumerate}
\end{defn}
For a Polish space $X$, we say a property, $P$, holds for comeagerly many $x\in X$ or $P$ holds for a generic $x\in X$ if the set of points in $X$ with property $P$ is comeager, that is, it contains an intersection of countably many dense open subsets of $X$.
\begin{defn}
Let $X,Y$ be two Polish spaces. We call a map $f:X\to Y$ \textbf{category preserving} if inverse image of a comeager subset of $Y$ is comeager in $X$. 
\end{defn}
Let $H$ be a separable complex Hilbert space. By $\mathcal{U}(H)$ we denote the group of all unitary operators on $H$. We consider $\mathcal{U}(H)$ with the strong topology, that is, pointwise convergence. This topology coincides with the weak topology on $\mathcal{U}(H)$ defined by
\[
T_k \rightarrow_w T \ \text{if for all} \  x, y \in H, \ \langle T_k(x), y\rangle \rightarrow \langle T(x), y\rangle.
\]
Note that $\mathcal{U}(H)$ with the strong topology is a Polish group. Let $u$ be a unitary operator on $H$ and $h\in H$. Then, there exists a unique measure (up to mutual absolute continuity), $\sigma_h$, on the unit circle, $\mathbb{T}$, such that for every integer number $k$
\[
\langle u^k(h),h\rangle=\int_\mathbb{T}{z^k} d\sigma_h.
\]
This measure is known as the spectral measure of the vector $h$. Let $C(h)$ be the closure of the set of all finite linear combinations of $\{u^k(h)\}_{k=1}^\infty$. Then, there are $\{h_i\}_{i=1}^\infty$ in $H$ such that
\begin{enumerate}
\item $C(h_i)\perp C(h_j)$ for $i\neq j$ and $H=\bigoplus C(h_i)$,
\item $\sigma_{h_{i+1}}\ll \sigma_{h_i}$ for every positive integer number $i$,
\item such sequence of measures, $\{\sigma_{h_i}\}_{i=1}^\infty$, is unique up to mutual absolute continuity.
\end{enumerate}
The maximal spectral type of $u$ is the largest measure among $\{\sigma_{h_i}\}_{i=1}^\infty$, that is, $\sigma_{h_1}$. Note that the maximal spectral type of $u$ is unique up to mutual absolute continuity. We denote the maximal spectral type of $u$ by $\sigma_u$. We encourage the reader to review Cornfeld--Fomin--Sinai [\ref{Sinai_B}, appendix 2] for more information on the maximal spectral type.
\bigskip\noindent

\section{Generic behavior of a function in $L^0 (\mu,\mathbb{T})$}
Let $(X,\mu)$ be a Borel probability measure. If $\alpha,\beta\in \mathbb{T}$ are independent over $\mathbb{Q}$, then non-zero powers of $\alpha$ and $\beta$ are distinct, that is, for a generic constant function in $L^0 (\mu,\mathbb{T})$ non-zero powers are distinct. In this section, we show that a similar property holds for a generic function in $L^0 (\mu,\mathbb{T})$. We use this property to prove the main theorem in the next section.\\

We need to introduce the following definitions before stating the main theorem of this section. In the following definition, we measure correlation (or intersection) between functions in $L^0 (\mu,\mathbb{T})$.
\begin{defn}\label{disjointness}
Let $\mu$ be a Borel probability measure on a standard Borel space $X$, and 
\[
{(f_1,\dots,f_m,g_1,\dots,g_n)\in L^0(\mu,\mathbb{T})^{m+n}}.
\]
We define
\[
\begin{split}
\rho_{m,n}&(f_1,\dots,f_m,g_1,\dots,g_n)\\
&=\operatorname{inf}\{\mu(A) \ : \ f_1(X\setminus A)\cdots f_m(X\setminus A)\cap g_1(X\setminus A)\cdots g_n(X\setminus A)=\emptyset\}.
\end{split}
\]
\end{defn}
Note that if $(f_1,\dots,f_m,g_1,\dots,g_n)$ and $(f'_1,\dots,f'_m,g'_1,\dots,g'_n)$ represent the same member of $L^0(\mu,\mathbb{T})^{m+n}$, then there exists $B\subseteq X$ such that $\mu(B)=0$, 
\[
f_i\upharpoonright (X\setminus B)=f'_i\upharpoonright (X\setminus B)
\]
for every $1\leq i\leq m$, and 
\[
g_j\upharpoonright (X\setminus B)=g'_j\upharpoonright (X\setminus B)
\]
for every $1\leq j\leq n$. Therefore, the value of 
\[
\rho_{m,n}(f_1,\dots,f_m,g_1,\dots,g_n)
\]
depends only on the class of $(f_1,\dots,f_m,g_1,\dots,g_n)$ in $L^0(\mu,\mathbb{T})^{m+n}$.

 Furthermore, when
\[
\rho_{m,n}(f_1,\dots,f_m,g_1,\dots,g_n)=0,
\]
we use the notation 
\[
f_1\cdots f_m\cap g_1\cdots g_n\approx\emptyset.
\]
\begin{lem}\label{rho_mn}
Let $\mu$ be a non-atomic Borel probability measure on a standard Borel space $X$, and 
\[
{(f_1,\dots,f_m,g_1,\dots,g_n)\in L^0(\mu,\mathbb{T})^{m+n}}.
\]
If 
\[
\rho_{m,n}(f_1,\dots,f_m,g_1,\dots,g_n)=0,
\]
then there exists $A\subseteq X$ with $\mu(A)=0$ such that
\[
f_1(X\setminus A)\cdots f_m(X\setminus A)\cap g_1(X\setminus A)\cdots g_n(X\setminus A)=\emptyset.
\]
\end{lem}
\begin{proof}
If 
\[
\rho_{m,n}(f_1,\dots,f_m,g_1,\dots,g_n)=0,
\]
then for every natural number $k$ we can find a Borel subset of $X$, $A_k$, such that $\mu(A_k)\leq \frac{1}{2^k}$ and
\[
f_1(X\setminus A_k)\cdots f_m(X\setminus A_k)\cap g_1(X\setminus A_k)\cdots g_n(X\setminus A_k)=\emptyset.
\]
Let 
\[
A=\bigcap_{i=1}^\infty\bigcup_{k\geq i}A_k.
\]
Then, $\mu(A)=0$ and for every $x_1,\dots,x_m,y_1,\dots,y_n\in X\setminus A$, there exists a natural number $k$ such that $x_1,\dots,x_m,y_1,\dots,y_n\in X\setminus A_k$. Thus,
\[
f_1(X\setminus A)\cdots f_m(X\setminus A)\cap g_1(X\setminus A)\cdots g_n(X\setminus A)=\emptyset.
\]
\end{proof}

\begin{thm}\label{T1}
Let $\mu$ be a non-atomic Borel probability measure on a standard Borel space $X$. Then, for every $m,n\in \mathbb{N}$, and sequences $1\leq k_1,\dots,k_p\leq m$, $1\leq l_1,\dots,l_q\leq n$
\[
\begin{split}
E=\{(f_1,\dots,f_m,g_1,\dots,g_n) &\in L^0(\mu,\mathbb{T})^{m+n} \ : \ f_{k_1}\cdots f_{k_p}\cap g_{l_1}\cdots g_{l_q}\approx\emptyset\}
\end{split}
\]
is comeager.
\end{thm}

We need the following lemma and proposition to prove Theorem \ref{T1}. 
\begin{lem}\label{lem1}
Let $\mu$ be a non-atomic Borel probability measure on a standard Borel space $X$. Then, for every $m,n\in \mathbb{N}$, and sequences $1\leq k_1,\dots,k_p\leq m$, $1\leq l_1,\dots,l_q\leq n$
\begin{align*}
\{(f_1,\dots,f_m,g_1,\dots,g_n)\in L^0(&\mu,\mathbb{T})^{m+n} \ :  \ f_1,\dots,f_m,g_1,\dots,g_n \ \text{are }\\ &\text{finite step  functions and} \ f_{k_1}\cdots f_{k_p}\cap g_{l_1}\cdots g_{l_q}\approx\emptyset\}
\end{align*}
is dense.
\end{lem}
\begin{proof}
We prove the lemma for $(m,n)=(p,q)=(2,1), \ k_1=l_1=1, \ k_2=2$ and the general statement follows with a similar argument. Set $\rho=\rho_{2,1}$ and let ${(f_0,g_0,h_0)\in L^0(\mu,\mathbb{T})^3}$. We can arbitrarily closely approximate $(f_0,g_0)$ with $(f,g)$ where $f,g$ are finite step functions. Since the range of $f$ and the range of $g$ are finite, $f(X)g(X)$ is finite. Therefore, we can find $h$ so that $h$ is a finite step function, $h$ is arbitrarily close to $h_0$, and $f(X)g(X)\cap h(X)=\emptyset$.
\end{proof}
\begin{prop}\label{P1}
Let $\mu$ be a non-atomic Borel probability measure on a standard Borel space $X$. Then, for every $k,m,n\in \mathbb{N}$, and sequences $1\leq k_1,\dots,k_p\leq m$, ${1\leq l_1,\dots,l_q\leq n}$
\begin{align*}
E_k=\{(f_1,\dots,f_m,g_1,\dots,g_n)\in L^0(\mu,&\mathbb{T})^{m+n} \ : \ \rho_{p,q}(f_{k_1},\dots,f_{k_p},g_{l_1},\dots,g_{l_q})\geq \frac{1}{k} \}
\end{align*}
is NWD.
\end{prop}

\begin{proof}
We prove the proposition for $(m,n)=(p,q)=(2,1), \ k_1=l_1=1, \ k_2=2$ and the general statement follows with a similar argument. Set $\rho=\rho_{2,1}$ and let $U\subseteq L^0(\mu,\mathbb{T})^3$ be an arbitrary open subset. By Lemma \ref{lem1}, we can find finite step functions $(f_0,g_0,h_0)\in U$ such that ${\rho(f_0,g_0,h_0)=0}$. We show that $\rho$ is continuous at $(f_0,g_0,h_0)$, that is, for every $\epsilon>0$ there is an open neighborhood of $(f_0,g_0,h_0)$, $V_\epsilon$, such that for every $(f,g,h)\in V_\epsilon$ we have $\rho(f,g,h)<\epsilon$. 

Fix $\epsilon>0$, we define
\[
V_\epsilon=\{(f,g,h)\in L^0(\mu,\mathbb{T})^3 \ : \ \int{\Bigl(\left|f-f_0\right|+\left|g-g_0\right|+\left|h-h_0\right|\Bigr)} d\mu < \epsilon^2\}.
\]
Note that $V_\epsilon$ is an open subset of $L^0(\mu,\mathbb{T})^3$ and $\{V_{\frac{1}{n}}\}_{n=1}^\infty$ is a basis for open neighborhoods of $(f_0,g_0,h_0)$ in $L^0(\mu,\mathbb{T})^3$. For $(f,g,h)\in V_\epsilon$, we define
\begin{align*}
&A=\{x \ : \ \left|f(x)-f_0(x)\right|>\epsilon\},\\
&B=\{x \ : \ \left|g(x)-g_0(x)\right|>\epsilon\},\\
&C=\{x \ : \ \left|h(x)-h_0(x)\right|>\epsilon\}.
\end{align*}
Note that since $(f,g,h)\in V_\epsilon$, 
\[
\mu(A\cup B\cup C)<\epsilon.
\]
For $x,y,z\in X\setminus (A\cup B\cup C)$, assuming $\epsilon$ is small enough, we have
\begin{align*}
\left| f(x)g(y)-h(z)\right|&\geq \left|f_0(x)g_0(y)-h_0(z)\right| - \left|h(z)-h_0(z)\right| - \left| (f(x)-f_0(x))g(y)\right| -\\ & \ \ - \left|f_0(x)(g(y)-g_0(y))\right|\\ &\geq \left| f_0(x)g_0(y)-h_0(z)\right| - \left| f(x)-f_0(x)\right| -\left| g(y)-g_0(y)\right|-\\ & \ \ -\left| h(z)-h_0(z)\right|\\ &\geq \left| f_0(x)g_0(y)-h_0(z)\right|-3 \epsilon>0.
\end{align*}
Note that since $\rho(f_0,g_0,h_0)=0$, $f_0(x)g_0(y)$ and $h_0(z)$ are distinct. Hence, if $\epsilon$ is small enough, then
\[
\left| f_0(x)g_0(y)-h_0(z)\right|-3 \epsilon>0.
\]
Therefore,
\[
f\Bigl(X\setminus (A\cup B\cup C)\Bigr)g\Bigl(X\setminus (A\cup B\cup C)\Bigr)\cap h\Bigl(X\setminus (A\cup B\cup C)\Bigr)=\emptyset.
\]
%and 
%\[
%\mu(A\cup B\cup C)<\epsilon.
%\]
Thus, putting the above equation and the fact that
\[
\mu(A\cup B\cup C)<\epsilon
\]
together we get that for every $(f,g,h)\in V_\epsilon$
\[
\rho(f,g,h)<\epsilon.
\]
Assuming $\epsilon$ is small enough, $V_\epsilon\subseteq U$ since $\{V_{\frac{1}{n}}\}_{n=1}^\infty$ is a basis for open neighborhoods of $(f_0,g_0,h_0)$, and $\rho$ is less than $\frac{1}{k}$ on $V_\epsilon$. Hence, $E_k$ is not dense in $U$. Since $U$ is an arbitrary open subset of $L^0(\mu,\mathbb{T})^3$, $E_k$ is NWD. 
\end{proof}
\begin{proof}[Proof of Theorem \ref{T1}]
Theorem follows from Proposition \ref{P1} since
\[
E^c=\bigcup_{k=1}^\infty E_k. \qedhere
\]
\end{proof}

It is noteworthy that a similar argument can be repeated to prove the following generalization of Theorem \ref{T1} for non-discrete Polish groups. %Recall that a topological space $G$ is perfect if every point $g\in G$ is a limit point of $G$. 
Note that if a Polish space $G$ is non-discrete, then every non-empty open subset of $G$ is infinite (uncountable).

\begin{thm}
Let $\mu$ be a non-atomic Borel probability measure on a standard Borel space $X$ and $G$ be a non-discrete Polish group. Then, for every $m,n\in \mathbb{N}$, and sequences $1\leq k_1,\dots,k_p\leq m$, $1\leq l_1,\dots,l_q\leq n$
\[
\begin{split}
E=\{(f_1,\dots,f_m,g_1,\dots,g_n) &\in L^0(\mu,G)^{m+n} \ : \ f_{k_1}\cdots f_{k_p}\cap g_{l_1}\cdots g_{l_q}\approx\emptyset\}
\end{split}
\]
is comeager.
\end{thm}

In the following, we define a notion of disjointness for a function $f\in L^0(\mu,\mathbb{T})$.
\begin{defn}
Let $\kappa=(u_1,\dots,u_t)$, $\kappa '=(v_1,\dots,v_{t'})$ be tuples in $\mathbb{Z}\setminus \{0\}$, and $R\subseteq S_t$. Recall definition of $f^\kappa$ from (\ref{function}). We say $f^\kappa$ and $f^{\kappa '}$ are \textbf{almost $R$-disjoint} if there exists $A\subseteq X$ such that $\mu(A)=0$ and for all 
\[
x_1,\dots,x_t,y_1,\dots,y_{t'}\in X\setminus A
\]
where $x_1,\dots,x_t$ are pairwise distinct and $y_1,\dots,y_{t'}$ are pairwise distinct, we have
\begin{align*}
f^\kappa(x_1,\dots,x_t)=f^{\kappa '}(y_1,\dots,y_{t'})\Rightarrow (y_1,\dots,y_{t'})=r(x_1,\dots,x_t) \ \ \text{for some} \ r\in R.
\end{align*}
Note that if $t\neq t'$ (or $R=\emptyset$), then
\[
f^\kappa(x_1,\dots,x_t)\neq f^{\kappa '}(y_1,\dots,y_{t'})
\]
given that there is no $r\in S_t$ (or $r\in R$) such that
\[
(y_1,\dots,y_{t'})=r(x_1,\dots,x_t).
\]
In particular, if $\kappa=\kappa '$, we say $f^\kappa$ is \textbf{almost $R$-to-one}.
\end{defn}

\begin{thm}\label{T2}
Let $\mu$ be a non-atomic Borel probability measure on a standard Borel space $X$. Let $\kappa=(u_1,\dots,u_t)$, $\kappa '=(v_1,\dots,v_{t'})$ be tuples in $\mathbb{Z}\setminus \{0\}$ and
\[
R=\{r\in S_t \ : \ (v_1,\dots,v_{t'})=r(u_1,\dots,u_t)\}.
\]
Then, for comeagerly many $f\in L^0(\mu,\mathbb{T})$, $f^\kappa$ and $f^{\kappa '}$ are almost $R$-disjoint.
\end{thm}
Note that in Theorem \ref{T2},  $R=\emptyset$ if $t\neq t'$. In general, $R$ could be any subgroup of $S_t$ moved by a fixed $r_0\in S_t$. For example, if $t=t'$ and all $u_i,v_j=1$, then $R=S_t$.
\begin{proof}
Let
\begin{align*}
\Omega=\{(\textbf{x},\textbf{y}) \in X^t\times X^{t'} \ : \ \text{coordinates of }&\textbf{x}\text{ are pairwise distinct and }\\ &\text{coordinates of }\textbf{y}\text{ are pairwise distinct}\}.
\end{align*}
Given $\textbf{x}=(x_1,\dots,x_t)$ and $\textbf{y}=(y_1,\dots,y_{t'})$ with $(\textbf{x},\textbf{y})\in \Omega$, there are unique $\textbf{i}=(i_1<\dots<i_w)$ and $\textbf{j}=(j_1<\dots<j_w)$ for some natural number $w$ such that
\begin{align*}
\{x_i \ : \ i\leq t\}\cap\{y_j \ : \ j\leq t'\}=\{x_{i_1},\dots,x_{i_w}\}=\{y_{j_1},\dots, y_{j_w}\}.
\end{align*}
Let $r\in S_w$ be the unique permutation such that
\[
(y_{j_1},\dots,y_{j_w})=r(x_{i_1},\dots,x_{i_w}),
\]
that is, $x_{i_k}=y_{j_{r(k)}}$ for every $1\leq k\leq w$.
%\[
%x_{i_k}=y_{i_{r(k)}} \ \text{for every } 1\leq k\leq w.
%\]
Let $m(\textbf{x},\textbf{y})=(\textbf{i},\textbf{j},r)$. We define
\begin{align*}
P_{\textbf{i},\textbf{j},r}:=\{ (\textbf{x},\textbf{y}) \in \Omega \ : \ m(\textbf{x},\textbf{y})=(\textbf{i},\textbf{j},r) \}.
\end{align*}
%We define:
%\begin{align*}
%P_{\textbf{i},\textbf{j},r}:=\{ (\textbf{x},\textbf{y}) \ : \ \textbf{x}=(x_1,\dots,x_t), \textbf{y}=(y_1,\dots,y_t) \ &\text{and} \ x_{i_1},\dots,x_{i_w}=r(y_{j_1},\dots,y_{j_w}) \\ &\text{where} \ \textbf{i}=(i_1,\dots,i_w), \textbf{j}=(j_1,\dots,j_w)\}
%\end{align*}
%for $0\leq w\leq t$ and $r\in S_w$. It is easy to see that:
Then,
\begin{align*}
\Omega=\bigcup_{\textbf{i},\textbf{j},r} P_{\textbf{i},\textbf{j},r}.
\end{align*}
Since there are finitely many such sets, it is enough to show that given $\textbf{i},\textbf{j},r$, for comeagerly many $f\in L^0(\mu,\mathbb{T})$ there exists $A\subseteq X$ with $\mu(A)=0$ such that for all $(\textbf{x},\textbf{y}) \in P_{\textbf{i},\textbf{j},r}$ where $x_i,y_j\in X\setminus A$ for every $1\leq i\leq t, 1\leq j\leq t'$, 
\[
\text{if} \ f^\kappa(x_1,\dots,x_t)=f^{\kappa '}(y_1,\dots,y_{t'}),\ \text{then} \ (y_1,\dots,y_{t'})=r(x_1,\dots,x_t).
\]

Fix $\textbf{i}=(i_1<\dots<i_w)$, $\textbf{j}=(j_1<\dots<j_w)$, and $r$. We have
\[
P_{\textbf{i},\textbf{j},r}=\bigcup_{U_1,\dots,U_t,V_1,\dots,V_{t'}} \{(\textbf{x},\textbf{y})\in \prod_{i=1}^{t} U_i\times\prod_{j=1}^{t'} V_j\ : \ m(\textbf{x},\textbf{y})=(\textbf{i},\textbf{j},r)\},
\]
where $U_i$, $1\leq i\leq t$, $V_j$, $j\in\{1,\dots,t'\}\setminus \{j_1,\dots,j_w\}$, are pairwise disjoint basic open subsets of $X$ and $U_{i_k}=V_{j_{r(k)}}$ for every $1\leq k\leq w$. Since there are only countably many different choices for $U_1,\dots,U_t,V_1,\dots,V_{t'}$, it is enough to show that if
\[
C=\{(\textbf{x},\textbf{y})\in \prod_{i=1}^{t} U_i\times\prod_{j=1}^{t'} V_j \ : \ m(\textbf{x},\textbf{y})=(\textbf{i},\textbf{j},r)\},
\]
then for comeagerly many $f\in L^0(\mu,\mathbb{T})$ there exists $A\subseteq X$ with $\mu(A)=0$ such that for all $(\textbf{x},\textbf{y}) \in C$ where $x_i,y_j\in X\setminus A$ for every $1\leq i\leq t, 1\leq j\leq t'$, 
\[
\text{if} \ f^\kappa(x_1,\dots,x_t)=f^{\kappa '}(y_1,\dots,y_{t'}),\ \text{then} \ (y_1,\dots,y_{t'})=r(x_1,\dots,x_t).
\]

Fix $C$ with sequences $U_i$, $1\leq i\leq t$, $V_j$, $1\leq j\leq t'$, as above. We may assume that $U_i$, $V_j$, $1\leq i\leq t$, $1\leq j\leq t'$, have non-zero measure since otherwise, we can take $A$ to be the union of all basic open subsets of $X$ with measure $0$. In this case, there is no $(\textbf{x},\textbf{y}) \in C$ with $x_i,y_j\in X\setminus A$ for every $1\leq i\leq t, 1\leq j\leq t'$.

 Consider the map 
\[
\Psi:L^0(\mu,\mathbb{T})\rightarrow \prod_{i=1}^t{L^0(U_i,\frac{\mu}{\mu(U_i)},\mathbb{T})}\times\prod_{j=1}^{t'-w}{L^0(V_{l_j},\frac{\mu}{\mu(V_{l_j})},\mathbb{T})}
\]
defined by
\begin{align*}
\Psi(f)=(f\upharpoonright U_1,\dots,f\upharpoonright U_t,f\upharpoonright &V_{l_1},\dots,f\upharpoonright V_{l_{t'-w}})
\end{align*}
where
\[
\{l_1,\dots,l_{t'-w}\}=\{1,\dots,t'\}\setminus \{j_1,\dots,j_w\}.
\]
The map $\Psi$ is open and continuous since for each open subset $U\subseteq X$ with $\mu(U)>0$, 
\[
f\rightarrow f\upharpoonright U
\]
is open and continuous, and $U_i$, $V_j$, $1\leq i\leq t$, $j\in\{1,\dots,t'\}\setminus \{j_1,\dots,j_w\}$, are pairwise disjoint. Therefore, $\Psi$ is category preserving. 

Let
\[
\Psi(f)=(f_1,\dots,f_t,g_{l_1},\dots,g_{l_{t'-w}})
\]
and
\[
(x_1,\dots,x_t,y_1,\dots,y_{t'})\in C.
\]
We have
\begin{align*}
&f^\kappa(x_1,\dots,x_t)=f_1(x_1)^{u_1}f_2(x_2)^{u_2}\cdots f_t(x_t)^{u_t},\\
&f^{\kappa '}(y_1,\dots,y_{t'})=\prod_{k=1}^{t'-w} g_{l_k}(y_{l_k})^{v_{l_k}}\cdot \prod_{k=1}^w f_{i_{r^{-1}(k)}}(x_{i_{r^{-1}(k)}})^{v_{j_k}}.
\end{align*}
If $w\neq t$ or in the case of $w=t$, $(v_1,\dots,v_{t'})\neq r(u_1,u_2,\dots,u_t)$, then by Theorem \ref{T1} and the fact that $\Psi$ is category preserving, for comeagerly many $f\in L^{0}(\mu,\mathbb{T})$
\begin{equation}\label{lastchapter3}
f_1^{u_1}f_2^{u_2}\cdots f_t^{u_t}\cap \prod_{k=1}^{t'-w} g_{l_k}^{v_{l_k}}\cdot \prod_{k=1}^w f_{i_{r^{-1}(k)}}^{v_{j_k}}\approx \emptyset.
\end{equation}
Note that since $U_i$, $1\leq i\leq t$, $V_j$, $1\leq j\leq t'$, have non-zero measures, $L^0(\frac{\mu}{\mu(U_i)},\mathbb{T})$, $1\leq i\leq t$, and $L^0(\frac{\mu}{\mu(V_j)},\mathbb{T})$, $1\leq j\leq t'$, are isomorphic to $L^0(\mu,\mathbb{T})$ as topological groups. Therefore, considering that $f\to f^{-1}$ is a category preserving map on $L^0(\mu,\mathbb{T})$, Theorem \ref{T1} can be applied to
\[
\prod_{i=1}^t{L^0(\frac{\mu}{\mu(U_i)},\mathbb{T})}\times\prod_{j=1}^{t'-w}{L^0(\frac{\mu}{\mu(V_{l_j})},\mathbb{T})}\cong L^0(\mu,\mathbb{T})^{t+t'-w}.
\]
to obtain (\ref{lastchapter3}). By Lemma \ref{rho_mn}, there exists $A\subseteq X$ with $\mu(A)=0$ such that for all $(\textbf{x},\textbf{y}) \in C$ where $x_i,y_j\in X\setminus A$ for every $1\leq i\leq t, 1\leq j\leq t'$,
\[
f^\kappa(x_1,\dots,x_t)=f^{\kappa '}(y_1,\dots,y_{t'})\Rightarrow (y_1,\dots,y_{t'})=r(x_1,\dots,x_t). \qedhere
\]
\end{proof}

With a similar argument one can prove the following generalization of Theorem \ref{T2} for non-discrete Polish groups.

\begin{thm}
Let $\mu$ be a non-atomic Borel probability measure on a standard Borel space $X$ and $G$ be a non-discrete Polish group. Let $\kappa=(u_1,\dots,u_t)$, $\kappa '=(v_1,\dots,v_{t'})$ be tuples in $\mathbb{Z}\setminus \{0\}$ and
\[
R=\{r\in S_t \ : \ (v_1,\dots,v_{t'})=r(u_1,\dots,u_t)\}.
\]
Then, for comeagerly many $f\in L^0(\mu,G)$, $f^\kappa$ and $f^{\kappa '}$ are almost $R$-disjoint.
\end{thm}

\bigskip\noindent

\section{DL--condition}
Generic behavior of a measure preserving transformation is of interest in Ergodic Theory. For instance, the following papers are devoted to study generic behavior of a measure preserving transformation: Del Junco--Lema\'nczyk [\ref{DL_B}], Glasner--Weiss [\ref{GW_B}], King [\ref{King_B}]. Of particular interest is characterization of $\overline{\langle T\rangle}$ for a generic $T\in L^0(\mu,\mathbb{T})$. More precisely, does there exists a topological group $G$ such that $\overline{\langle T\rangle}$ is isomorphic to $G$ for a generic $T\in L^0(\mu,\mathbb{T})$. The following question is due to Glasner and Weiss.

\begin{ques}[Glasner--Weiss]\label{q1}
Let $\mu$ be a non-atomic Borel probability measure on $X$. Is it true that for a generic $T\in \operatorname{Aut}(\mu)$
\[
\overline{\langle T\rangle}\cong L^0(\mu,\mathbb{T})?
\]
 %In particular, for a generic measure preserving transformation, $T$, $\overline{\langle T\rangle}$ is L$\acute{e}$vy .
\end{ques}
Del Junco and Lema\'nczyk [\ref{DL_B}] proved a generic property of measure preserving transformations. They showed that for a generic $T\in \operatorname{Aut}(\mu)$, maximal spectral types of powers of $T$ satisfy certain orthogonality conditions.
\begin{thm}[Del Junco--Lema\'nczyk]\label{DL}
	There is a dense $G_\delta$ subset $G\subseteq \operatorname{Aut}(\mu)$ such that, for each $T\in G$ and $k(1), k(2), \dots, k(l)\in \mathbb{Z}^+$, $k'(1), k'(2), \dots, k'(l')\in \mathbb{Z}^+$, the convolutions
	\[
	\sigma_{T^{k(1)}}\ast\cdots\ast\sigma_{T^{k(l)}} \ \text{and} \ \sigma_{T^{k'(1)}}\ast\cdots\ast\sigma_{T^{k'(l')}}
	\]
	are mutually singular, provided that there does not exist $r\in S_l$ such that 
	\[
	(k'(1), k'(2), \dots, k'(l'))=r(k(1), k(2), \dots, k(l)).
	\]
	\end{thm}
Assuming the answer to Question \ref{q1} is positive, for a generic $T\in \operatorname{Aut}(\mu)$, we can define a continuous unitary representation of $L^0(\mu,\mathbb{T})$ by 
\[
\Psi:L^0(\mu,\mathbb{T}) \cong \overline{\langle T\rangle} \hookrightarrow \operatorname{Aut}(\mu)\subseteq \mathcal{U}(L^2(X,\mu)).
\]
Note that $\operatorname{Aut}(\mu)$ can be viewed as a closed subset of $\mathcal{U}(L^2(X,\mu))$ by identifying ${T\in \operatorname{Aut}(\mu)}$ with $U_T\in \mathcal{U}(H)$ where
\[
 U_T(f)=f\circ T^{-1}.
\]
Hence, the orthogonality conditions from Theorem \ref{DL} motivates the following orthogonality conditions for a continuous unitary representation of $L^0(\mu,\mathbb{T})$.
\begin{defn}
   Fix a non-atomic Borel probability measure $\mu$ on a standard Borel space $X$. We say that a continuous unitary representation of $L^0(\mu,\mathbb{T})$,
\[
\Phi:L^0(\mu,\mathbb{T})\rightarrow \mathcal{U}(H)
\]
satisfies the \textbf{DL--condition} if there is a dense $G_\delta$ subset $G\subseteq L^0(\mu,\mathbb{T})$ such that, for each $f\in G$ and $k(1), k(2), \dots, k(l)\in \mathbb{Z}^+$, $k'(1), k'(2), \dots, k'(l')\in \mathbb{Z}^+$, the convolutions
	\[
	\sigma_{\Phi(f)^{k(1)}}\ast\cdots\ast\sigma_{\Phi(f)^{k(l)}} \ \text{and} \ \sigma_{\Phi(f)^{k'(1)}}\ast\cdots\ast\sigma_{\Phi(f)^{k'(l')}}
	\]
	are mutually singular, provided that there does not exist $r\in S_l$ such that 
	\[
	(k'(1), k'(2), \dots, k'(l'))=r(k(1), k(2), \dots, k(l)).
	\]
\end{defn}
Solecki [\ref{S1_B}] showed that a continuous unitary representation of $L^0(\mu,\mathbb{T})$ can be written as a direct sum of unitary representations of $L^0(\mu,\mathbb{T})$ of the following form: Assume that we are given a sequence $\kappa=(k(1),k(2),\dots,k(n))$ of elements of $\mathbb{Z}\setminus \{0\}$ with
\[
k(1)\leq k(2)\leq\cdots\leq k(n).
\]
Assume we have a finite Borel measure $\lambda$ on $X^n$ whose marginal measures are absolutely continuous with respect to $\mu$, that is, for $i\leq n$
\begin{equation}\label{A1}
 (\pi_i)_*(\lambda)\ll\mu 
\end{equation}
where $\pi_i$ is the projection on the $i-$th coordinate. With this set of data we associate the following representation of $L^0(\mu,\mathbb{T})$ on $L^2(\lambda,\mathbb{C})$
\[
L^0(\mu,\mathbb{T})\ni f\rightarrow U_f\in \mathcal{U}(L^2(\lambda,\mathbb{C}))
\]
where for $h\in L^2(\lambda,\mathbb{C})$
\[
U_f(h)=(\prod_{i\leq n} (f \circ\pi_i)^{k(i)}) h.
\]
This representation is denoted by $\sigma(\kappa,\lambda)$. Furthermore, we consider the following additional condition on a finite measure $\lambda$ as above, for $1\leq i<j\leq n$
\begin{equation}\label{A2}
\lambda(\{(x_1,x_2,\dots,x_n)\in X^n \ : \ x_i=x_j\})=0 .
\end{equation}
Let $S$ be the set of all sequences $\kappa=(k(1),k(2),\dots,k(n))$ of elements of $\mathbb{Z}\setminus \{0\}$ such that $k(1)\leq k(2)\leq\cdots\leq k(n)$. The natural number $n$ is called the length of $\kappa$ and we denote it by $\left|\kappa\right|$.
\begin{thm}[Solecki]\label{Theorem_S}
Let $\Phi$ be a continuous unitary representation of $L^0(\mu,\mathbb{T})$ on a separable complex Hilbert space $H$. Consider $H_0$, the orthogonal complement of
\[
\{v\in H \ : \ \forall f\in L^0(\mu,\mathbb{T}) \ \Phi(f)(v)=v \}.
\]
For $\kappa\in S$ and $i\in \mathbb{N}$ there exist finite Borel measures $\lambda_\kappa^i$ on $X^{|\kappa|}$ with properties \ref{A1}, \ref{A2}, and 
\begin{equation}\label{A3}
\lambda_\kappa^j\ll \lambda_\kappa^i \ \text{for} \ i<j
\end{equation}
such that the representation $\Phi$ restricted to $H_0$ is the direct sum of the representations $\sigma(\kappa,\lambda_\kappa^i)$ with $\kappa\in S$ and $i\in \mathbb{N}$. 
\end{thm}
Furthermore, these measures, $\lambda_\kappa^i$, can be chosen such that if $\kappa=(k(1),\dots,k(n)),\break m\in \mathbb{N},$ and $1\leq i<j\leq n$ with $k(i)=k(j)$, then
\begin{equation}\label{measure}
\lambda_\kappa^m\{(x_1,x_2,\dots,x_n)\in X^n \ : \ x_j<_X x_i\}=0.
\end{equation}
Here $<_X$ is a linear order on $X$ with the property that the order topology it generates is compact, second countable and the Borel sets with
respect to this topology coincide with the Borel sets on $X$. Assuming (\ref{measure}), measures $\lambda_\kappa^i$ obtained from Theorem \ref{Theorem_S} are unique up to mutual absolute continuity.

In the following, we show that the DL--condition for a continuous representation of $L^0({\mu,\mathbb{T}})$ is equivalent to certain orthogonality conditions on the measures, $\lambda_\kappa^1$.

\begin{thm}\label{MT}
Let $\Phi=\bigoplus\sigma(\kappa,\lambda_\kappa^i)$ be a continuous unitary representation of $L^0(\mu,\mathbb{T})$. Then, $\Phi$ satisfies the DL--condition iff we have
\begin{equation}\label{*}
\lambda_{\kappa_1}^1\times\lambda_{\kappa_2}^1\times\cdots\times\lambda_{\kappa_l}^1 \perp r(\lambda_{\kappa_{1}'}^1\times\lambda_{\kappa_{2}'}^1\times\dots\times\lambda_{\kappa_{l'}'}^1)
\end{equation}
for every $(\kappa_1, \kappa_2, \dots, \kappa_l)$, $(\kappa_1', \kappa_2', \dots, \kappa_{l'}')$, and $r\in \operatorname{S}_t$ such that
\[
k(1)\kappa_1+k(2)\kappa_2+\cdots+k(l)\kappa_l= r\Big(k'(1)\kappa_{1}'+k'(2)\kappa_{2}'+\cdots+k'(l')\kappa_{l'}'\Big) 
\]
for some non-zero integer numbers $(k(1), k(2), \dots, k(l))$ and $(k'(1), k'(2), \dots, k'(l'))$, provided that there does not exist $s\in S_l$ such that 
	\[
	(k'(1), k'(2), \dots, k'(l'))=s(k(1), k(2), \dots, k(l)).
	\]
\end{thm}

\bigskip\noindent

 %Note that since for a generic $T\in \operatorname{Aut}(\mu)$ at most one of the measures $\{\lambda_\kappa^i\}_{i=1}^\infty$, $\lambda_\kappa^1$, is non-zero, we may assume that  We use Definition \ref{r} and Definition \ref{k} in the statement of the following theorem.

\section{Proof of Theorem \ref{MT}}
We prove the theorem under the assumption that for every $\kappa$ at most one of the measures $\lambda_\kappa^i$ is non-zero, that is, only $\lambda_\kappa=\lambda_\kappa^1$ can be non-zero. Then, we show that the general statement follows consequently.

 Assume that $\Phi=\bigoplus\sigma(\kappa,\lambda_\kappa)$ satisfies the DL--condition. We will show that $\Phi$ satisfies (\ref{*}). Let $(\kappa_1, \kappa_2, \dots, \kappa_l)$, $(\kappa_1', \kappa_2', \dots, \kappa_{l'}')$, and $r\in \operatorname{S}_t$ be such that ($t=\left|\kappa_1\right|+\cdots+\left|\kappa_l\right|=\left|\kappa_1'\right|+\cdots+\left|\kappa_{l'}'\right|$)
\[
k(1)\kappa_1+k(2)\kappa_2+\cdots+k(l)\kappa_l= r\Big(k'(1)\kappa_{1}'+k'(2)\kappa_{2}'+\cdots+k'(l')\kappa_{l'}'\Big)
\]
for some non-zero integer numbers $(k(1), k(2), \dots, k(l))$ and $(k'(1), k'(2), \dots, k'(l'))$, provided that one is not a rearrangement of the other. For $f\in L^0(\mu,\mathbb{T})$, we define $h,h':X^t\rightarrow \mathbb{T}$ as follows:
\begin{align*}
h(x_1,\dots,x_t)=f^{k(1)\kappa_1+k(2)\kappa_2+\cdots+k(l)\kappa_l}(x_1,\dots,x_t),\\ h'(x_1,\dots,x_t)=f^{k'(1)\kappa_{1}'+k'(2)\kappa_{2}'+\cdots+k'(l')\kappa_{l'}'}(x_1,\dots,x_t).
\end{align*}
By the DL--condition, for comeagerly many $f\in L^0(\mu,\mathbb{T})$ we have
\[
\sigma_{\Phi(f)^{k(1)}}\ast\cdots\ast\sigma_{\Phi(f)^{k(l)}} \perp\sigma_{\Phi(f)^{k'(1)}}\ast\cdots\ast\sigma_{\Phi(f)^{k'(l')}}.
\]
One can show that the maximal spectral type of $\Phi(f)$ is equivalent to $\Sigma_\kappa \alpha_\kappa \mu_{f,\kappa}$ where $\mu_{f,\kappa}$ is the maximal spectral type of $\sigma(\kappa,\lambda_\kappa)(f)$ and $0<\alpha_\kappa<1$ is chosen so that $\Sigma_\kappa \alpha_\kappa \mu_{f,\kappa}$ is finite. Therefore, for comeagerly many $f\in L^0(\mu,\mathbb{T})$
\begin{equation}\label{eq}
\mu_{f,k(1)\kappa_1}\ast\cdots\ast\mu_{f,k(l)\kappa_l}\perp\mu_{f,k'(1)\kappa_1'}\ast\cdots\ast\mu_{f,k'(l')\kappa_{l'}'}.
\end{equation}
Furthermore, $\mu_{f,k(1)\kappa_1}\ast\cdots\ast\mu_{f,k(l)\kappa_l}$ is the push-forward measure of $\lambda_{\kappa_1}\times\cdots\times\lambda_{\kappa_l}$ under $h$, and $\mu_{f,k'(1)\kappa_1'}\ast\cdots\ast\mu_{f,k'(l')\kappa_{l'}'}$ is the push-forward measure of $\lambda_{\kappa_{1}'}\times\dots\times\lambda_{\kappa_{l'}'}$ under $h'$. Since $h=r(h')$, equation (\ref{eq}) indicates that the push-forward measure of $\lambda_{\kappa_1}\times\cdots\times\lambda_{\kappa_l}$ and $r(\lambda_{\kappa_{1}'}\times\dots\times\lambda_{\kappa_{l'}'})$ under the same function, $h=r(h')$, are orthogonal to each other. Hence,
\[
\lambda_{\kappa_1}\times\lambda_{\kappa_2}\times\cdots\times\lambda_{\kappa_l}\perp r(\lambda_{\kappa_{1}'}\times\lambda_{\kappa_{2}'}\times\dots\times\lambda_{\kappa_{l'}'}).
\]
Note that the same argument proves the general case where $\lambda_\kappa^i$ is not necessarily zero for $i>1$.

Now assume $\Phi$ satisfies (\ref{*}), we show that $\Phi$ also satisfies the DL--condition. By (\ref{*}),
\[ 
\lambda_{\kappa_1}\times\lambda_{\kappa_2}\times\cdots\times\lambda_{\kappa_l} \perp r(\lambda_{\kappa_{1}'}\times\lambda_{\kappa_{2}'}\times\dots\times\lambda_{\kappa_{l'}'})
\]
for every $(\kappa_1, \kappa_2, \dots, \kappa_l)$, $(\kappa_1', \kappa_2', \dots, \kappa_{l'}')$, and $r\in \operatorname{S}_t$ so that
\[
k(1)\kappa_1+k(2)\kappa_2+\cdots+k(l)\kappa_l= r(k'(1)\kappa_{1}'+k'(2)\kappa_{2}'+\cdots+k'(l')\kappa_{l'}') 
\]
for some non-zero integer numbers $(k(1), k(2), \dots, k(l))$ and $(k'(1), k'(2), \dots, k'(l'))$, provided that one is not a rearrangement of the other. We have
	\[
	\sigma_{\Phi(f)^{k(1)}}\ast\cdots\ast\sigma_{\Phi(f)^{k(l)}} \ \perp \ \sigma_{\Phi(f)^{k'(1)}}\ast\cdots\ast\sigma_{\Phi(f)^{k'(l')}} 
	\]
if and only if
	\[
	\mu_{f,k(1)\kappa_1}\ast\cdots\ast\mu_{f,k(l)\kappa_l}\perp\mu_{f,k'(1)\kappa_1'}\ast\cdots\ast\mu_{f,k'(l')\kappa_{l'}'}
	\]
for every $(\kappa_1, \kappa_2, \dots, \kappa_l)$ and $(\kappa_1', \kappa_2', \dots, \kappa_{l'}')$.

Fix $(\kappa_1, \kappa_2, \dots, \kappa_l)$ and $(\kappa_1', \kappa_2', \dots, \kappa_{l'}')$. We define $h:X^t\rightarrow\mathbb{T}$ and $h':X^{t'}\rightarrow\mathbb{T}$ as follows
\begin{align*}
h(x_1,\dots,x_t)=f^{k(1)\kappa_1+k(2)\kappa_2+\cdots+k(l)\kappa_l}(x_1,\dots,x_t),\\ h'(x_1,\dots,x_{t'})=f^{k'(1)\kappa_{1}'+k'(2)\kappa_{2}'+\cdots+k'(l')\kappa_{l'}'}(x_1,\dots,x_{t'})
\end{align*}
where
\begin{align*}
t&=\left|\kappa_1\right|+\left|\kappa_2\right|+\cdots+\left|\kappa_l\right|,\\ t'&=\left|\kappa_{1}'\right|+\left|\kappa_{2}'\right|+\cdots+\left|\kappa_{l'}'\right|.
\end{align*}	
Then, $\mu_{f,k(1)\kappa_1}\ast\cdots\ast\mu_{f,k(l)\kappa_l}$ and $\mu_{f,k'(1)\kappa_1'}\ast\cdots\ast\mu_{f,k'(l')\kappa_{l'}'}$ are push-forward measures of $\lambda_{\kappa_1}\times\cdots\times\lambda_{\kappa_l}$ and $\lambda_{\kappa_{1}'}\times\dots\times\lambda_{\kappa_{l'}'}$ under $h$ and $h'$, respectively. Assume there is $r_0\in S_t$ such that (in particular $t=t'$)
\[
k(1)\kappa_1+k(2)\kappa_2+\cdots+k(l)\kappa_l= r_0(k'(1)\kappa_{1}'+k'(2)\kappa_{2}'+\cdots+k'(l')\kappa_{l'}').
\]	
We define
\[
R=\{ r\in S_t : k(1)\kappa_1+k(2)\kappa_2+\cdots+k(l)\kappa_l=r(k(1)\kappa_1+k(2)\kappa_2+\cdots+k(l)\kappa_l)\}.
\]
Since $\Phi$ satisfies (\ref{*}), for every $r\in R$, we have
\[
\lambda_{\kappa_1}\times\lambda_{\kappa_2}\times\cdots\times\lambda_{\kappa_l} \perp r r_0(\lambda_{\kappa_{1}'}\times\lambda_{\kappa_{2}'}\times\dots\times\lambda_{\kappa_{l'}'}).
\]
Therefore, we can find $F,G\subseteq X^t$ such that 
\begin{align*}
&\lambda_{\kappa_1}\times\lambda_{\kappa_2}\times\cdots\times\lambda_{\kappa_l}(F)=1, \\ &r r_0(\lambda_{\kappa_{1}'}\times\lambda_{\kappa_{2}'}\times\dots\times\lambda_{\kappa_{l'}'})(r(G))=1 \ \text{for every} \ r\in R, \ 
\end{align*}
and $F\cap r(G)=\emptyset$ for every $r\in R$. By Theorem \ref{T2}, for comeagerly many ${f\in L^0(\mu,\mathbb{T})}$ the push-forward measures of $\lambda_{\kappa_1}\times\cdots\times\lambda_{\kappa_l}$ and $r_0(\lambda_{\kappa_{1}'}\times\dots\times\lambda_{\kappa_{l'}'})$ under $f^{k(1)\kappa_1+\cdots+k(l)\kappa_l}$ are perpendicular to each other since the latter holds for every ${f\in L^0(\mu,\mathbb{T})}$ such that $f^{k(1)\kappa_1+\cdots+k(l)\kappa_l}$ is almost $R$-to-one. Therefore,
	\[
	\mu_{f,k(1)\kappa_1}\ast\cdots\ast\mu_{f,k(l)\kappa_l}\perp\mu_{f,k'(1)\kappa_1'}\ast\cdots\ast\mu_{f,k'(l')\kappa_{l'}'}.
	\]	
If there does not exist such $r_0\in S_t$, then by Theorem \ref{T2}, for comeagerly many $f\in L^0(\mu,\mathbb{T})$
\[
f^{k(1)\kappa_1+\cdots+k(l)\kappa_l}\cap f^{k'(1)\kappa_{1}'+\cdots+k'(l')\kappa_{l'}'}\approx \emptyset.
\]
Thus, for comeagerly many ${f\in L^0(\mu,\mathbb{T})}$ the push-forward measures of $\lambda_{\kappa_1}\times\cdots\times\lambda_{\kappa_l}$ and $\lambda_{\kappa_{1}'}\times\dots\times\lambda_{\kappa_{l'}'}$ under $f^{k(1)\kappa_1+\cdots+k(l)\kappa_l}$ and $f^{k'(1)\kappa_{1}'+\cdots+k'(l')\kappa_{l'}'}$, respectively, are perpendicular to each other. Hence,
	\begin{equation}\label{mu}
	\mu_{f,k(1)\kappa_1}\ast\cdots\ast\mu_{f,k(l)\kappa_l}\perp\mu_{f,k'(1)\kappa_1'}\ast\cdots\ast\mu_{f,k'(l')\kappa_{l'}'}.
	\end{equation}
Note that the general case where $\lambda_\kappa^i$ is not necessarily zero for $i>1$ follows from (\ref{mu}) since for Borel measures 
\[
\mu_1,\dots,\mu_m,\nu_1,\dots,\nu_n,\mu_1',\dots,\mu_m',\nu_1',\dots,\nu_n'
\]
where $\mu_i'\ll \mu_i$ and $\nu_j'\ll \nu_j$, for $1\leq i\leq m$ and $1\leq j\leq n$, we have
\[
\text{if }\mu_1\ast\cdots\ast\mu_m \ \perp \ \nu_1\ast\cdots\ast\nu_n, \text{ then }\mu_1'\ast\cdots\ast\mu_m' \ \perp \ \nu_1'\ast\cdots\ast\nu_n'.
\]

\bigskip\noindent

\section{Generic behavior of a unitary operator}
Let $H$ be a separable infinite dimensional Hilbert space and $\psi\in H$ be a vector of length $1$. Melleray--Tsankov [\ref{MT_B},Theorem 4.4] proved that for a generic $u\in \mathcal{U}(H)$, $\overline{\langle u\rangle}$ is isomorphic to $L^0(\mu_u,\mathbb{T})$, where $\mu_u$ is the spectral measure of the vector $\psi$ with respect to $u$. Furthermore, they showed that the representation of $L^0(\mu_u,\mathbb{T})$ obtained by this isomorphism
\[
\Phi : L^0(\mu_u,\mathbb{T})\rightarrow \mathcal{U}(H)
\]
is the standard representation by pointwise multiplication, that is, for every $f\in L^0(\mu_u,\mathbb{T})$ we have
\[
\Phi(f)(h)=f\cdot h.
\]
Since for a generic $u\in \mathcal{U}(H)$, $\mu_u$ is non-atomic [\ref{MT_B},Lemma 4.3], we have that
\[
L^0(\mu_u,\mathbb{T})\cong L^0(\mu,\mathbb{T}).
\]

In the following, we use this result of Melleray--Tsankov [\ref{MT_B}] and Theorem \ref{MT} to show orthogonality conditions for a generic $u\in \mathcal{U}(H)$ analogous to orthogonality conditions in Theorem \ref{DL}. %proved by Del Junco--Lema\'nczyk [\ref{DL_B}].
\begin{cor}\label{Cor_1}
Let $H$ be a separable infinite dimensional Hilbert space. Then, for a generic $u\in \mathcal{U}(H)$, the convolutions
\[
	\sigma_{u^{k(1)}}\ast\cdots\ast\sigma_{u^{k(l)}} \ \text{and} \ \sigma_{u^{k'(1)}}\ast\cdots\ast\sigma_{u^{k'(l')}}
\]
are mutually singular, provided that there does not exist $r\in S_l$ such that 
	\[
	(k'(1), k'(2), \dots, k'(l'))=r(k(1), k(2), \dots, k(l)).
	\] 	
\end{cor}
We need the following theorem to prove Corollary \ref{Cor_1}.
\begin{thm}\label{KU-var}
Let $H$ be a separable infinite dimensional Hilbert space and $G$ be a subset of $\mathcal{U}(H)$ with the Baire property. Then, $G$ is comeager iff $G\cap\overline{\langle u\rangle}$ is comeager in $\overline{\langle u\rangle}$ for comeagerly many $u\in \mathcal{U}(H)$. 
\end{thm}
Solecki [\ref{S2_B}, Lemma 3] proved an analogous statement for $G\subseteq \operatorname{Aut}(\mu)$ with the Baire property. Assuming the following two lemmas (Lamma \ref{Lem1},\ref{Lem2}), the same proof can be repeated to prove Theorem \ref{KU-var}.
\begin{lem}\label{Lem1}
Let $H$ be a separable infinite dimensional Hilbert space. We define 
\[
f_n:\mathcal{U}(H)\rightarrow \mathcal{U}(H)
\]
such that
\[
f_n(u)=u^n.
\]
Then, $f_n$ is category preserving.
\end{lem}

\begin{proof}
It is enough to show that if $A\subseteq \mathcal{U}(H)$ is dense and open, then $f_n^{-1}(A)$ is dense. Let
\[
W=\{u\in \mathcal{U}(H) \ : \ \|u(Z_i)-u_0(Z_i)\|<\epsilon \ \text{for every } 1\leq i\leq k \}
\]
for some $u_0\in \mathcal{U}(H)$, $Z_i\in H$ for $i= 1,2,\dots,k$, and $\epsilon>0$. We have to find $\tilde{u}\in W$ such that $\tilde{u}^n\in A$. Let 
\[
H_0=\operatorname{span}\{Z_1,\dots,Z_k,u_0(Z_1),\dots,u_0(Z_k)\}.
\]
By modifying $u_0$ on $\operatorname{span}\{Z_1,\dots,Z_k\}^\perp$, we may assume that $u_0(H_0)=H_0$ and $u_0\upharpoonright H_0^\perp=id$. Since $u_0\upharpoonright H_0$ is a unitary operator of $H_0$, it is diagonalizable. Thus, we can find an orthonormal basis for $H_0$, $\{h_1,\dots,h_m\}$, and eigenvalues $\lambda_1,\dots,\lambda_m\in\mathbb{T}$ such that
\[
u_0(h_i)=\lambda_i h_i \ \ \ \ \text{for } i=1,2,\dots,m.
\]
Furthermore, we may assume that $\{(\lambda_1^{ni},\dots,\lambda_m^{ni})\}_{i=1}^\infty$ is dense in $\mathbb{T}^m$. 

Fix $\delta>0$ and let $l\in \mathbb{N}$ be such that 
\[
\sum_{i=1}^m{\|\lambda_i^{nl}-\lambda_i\|}< \delta.
\]
Then, $\|u_0^{nl}-u_0\|< \delta$. Since $A$ is open and dense, there exists $u_1\in A$ and a finite dimensional subspace $H_1\supseteq H_0$ such that $u_1(H_1)=H_1$ and 
\[
\|u_1(h_i)-u_0^n(h_i)\|< \delta \ \ \ \ \text{for } i=1,2,\dots,m.
\]
By modifying $u_1$ on $H_1^\perp$, we may assume that there exists an orthonormal basis for $H$, $\{h'_i\}_{i=1}^\infty$, and eigenvalues, $\{\lambda'_i\}_{i=1}^\infty$, such that for every $i\in \mathbb{N}$
\[
u_1(h'_i)=\lambda'_i h'_i.
\]
Note that since $u_1\upharpoonright H_1$ is a unitary operator of $H_1$, it is possible to find such basis for $H$. Moreover, by changing the order of $\{h'_i\}_{i=1}^\infty$ and further modifying $u_1$ on $H_1^\perp$, we may assume that $H_0\subseteq \operatorname{span}\{h'_1,\dots,h'_{nm}\}$ and for every $i\in \mathbb{N}$
\[
\lambda'_{ni-n+1}=\cdots=\lambda'_{ni}.
\]  
We define $\tilde{u}\in \mathcal{U}(H)$ such that
\[
\tilde{u}(h'_{an+b})=\begin{cases} u_1^l(h'_{an+b+1}) &\mbox{if } 1\leq b\leq n-1 \\ 
u_1^{1-(n-1)l}(h'_{an+1}) & \mbox{if } b=n \end{cases}.
\]
Then, $\tilde{u}^n=u_1$. Moreover, if $\delta$ is small enough, then for every $i=1,2,\dots,k$
%\[
%\|u_1(Z_i)-u_0^n(Z_i)\|< \epsilon \ \text{and} \ \|u_0(Z_i)-u_0^{nl}(Z_i)\|< \epsilon.
%\]
%Therefore, for every $i=1,2,\dots,k$
\[
\|u_1^l(Z_i)-u_0(Z_i)\|< \epsilon \ \text{and} \ \|u_1^{1-(n-1)l}(Z_i)-u_0(Z_i)\|< \epsilon.
\]
Thus, for every $i=1,2,\dots,k$
\[
\|\tilde{u}(Z_i)-u_0(Z_i)\|<\epsilon,
\]
that is, $\tilde{u}\in W$.
\end{proof}

\begin{lem}\label{Lem2}
Let $H$ be a separable infinite dimensional Hilbert space. Then, the set of all $u\in \mathcal{U}(H)$ with
\[
\{v u v^{-1} \ : \ v\in \mathcal{U}(H)\} \text{ dense}
\]
is dense.
\end{lem} 
We encourage the reader to review [\ref{Nad_B}, Proposition $8.23$] for a proof of Lemma \ref{Lem2}.
%\begin{proof}
%We say a unitary operator $u\in \mathcal{U}(H)$ is diagonalizable if there is a basis of $H$, $\{h_i\}_{i=1}^\infty$, and eigenvalues, $\{\lambda_i\}_{i=1}^\infty$, such that for every $i\in \mathbb{N}$
%\[
%u(h_i)=\lambda_i h_i.
%\]
%Let
%\begin{align*}
%D=\{u\in \mathcal{U}(H) : u \text{ is diagonalizable and for}&\text{ every } N\in \mathbb{N},\\ &\{\lambda_i\}_{i=N}^\infty \text{ is dense in } \mathbb{T}\}.
%\end{align*}
%We claim that $D$ is dense and for every $u\in D$
%\[
%\{v u v^{-1} \ : \ v\in \mathcal{U}(H)\}
%\]
%is dense. Let
%\[
%W=\{u\in \mathcal{U}(H) \ : \|u(Z_i)-u_0(Z_i)\|<\epsilon \ \text{for every } 1\leq i\leq k  \}
%\]
%for some $u_0\in \mathcal{U}(H)$, $Z_i\in H$ for $i= 1,2,\dots,k$, and $\epsilon>0$. Let 
%\[
%H_0=\operatorname{span}\{Z_1,\dots,Z_k,u_0(Z_1),\dots,u_0(Z_k)\}.
%\]
%By modifying $u_0$ on $\operatorname{span}\{Z_1,\dots,Z_k\}^\perp$, we may assume that $u_0(H_0)=H_0$. By modifying $u_0$ on $H_0^\perp$, we can find $u\in W\cap D$. Furthermore, assume $\lambda_1,\dots,\lambda_m$ are eigenvalues of $u_0\upharpoonright H_0$, $h_1,\dots,h_m$ are corresponding eigenvectors of $u_0\upharpoonright H_0$, and $u\in D$. Then, there is an orthonormal set of vectors, $\{h'_1,\dots,h'_m\}$, such that for every $i=1,2,\dots,m$
%\[
%u(h'_i)=\lambda'_i h'_i
%\]
%where $\left|\lambda_i-\lambda'_i\right|\ll \epsilon$. Then, there is a unitary operator, $v$, that sends $(h_1,\dots,h_m)$ to $(h'_1,\dots,h'_m)$. It is easy to see that $v^{-1} u v \in W$.
%\end{proof}

\begin{proof}[Proof of Corollary \ref{Cor_1}]
By [\ref{MT_B},Theorem 4.4], for a generic $u\in \mathcal{U}(H)$, the representation of $L^0(\mu,\mathbb{T})$ obtained by $\overline{\langle u\rangle}\cong L^0(\mu,\mathbb{T})$ is equal to $\sigma(1,\mu_u)$, that is, the representation has only one non-zero measure, namely, $\lambda_1^1=\mu_u$. Therefore, by Theorem \ref{MT}, the representation satisfies the DL--condition. Thus, for a generic $u\in \mathcal{U}(H)$, for a generic $v\in\overline{\langle u\rangle}$, we have: for every $k(1), k(2), \dots, k(l)\in \mathbb{Z}^+$, $k'(1), k'(2), \dots, k'(l')\in \mathbb{Z}^+$, the convolutions
	\[
	\sigma_{v^{k(1)}}\ast\cdots\ast\sigma_{v^{k(l)}} \ \text{and} \ \sigma_{v^{k'(1)}}\ast\cdots\ast\sigma_{v^{k'(l')}}
	\]
	are mutually singular, provided that there does not exist $r\in S_l$ such that 
	\[
	(k'(1), k'(2), \dots, k'(l'))=r(k(1), k(2), \dots, k(l)).
	\] 	
	
Fix $k(1), k(2), \dots, k(l)\in \mathbb{Z}^+$, $k'(1), k'(2), \dots, k'(l')\in \mathbb{Z}^+$, where there does not exist $r\in S_l$ such that 
	\[
	(k'(1), k'(2), \dots, k'(l'))=r(k(1), k(2), \dots, k(l)).
	\] 	
Let 
\[
G=\{u\in \mathcal{U}(H) \ : \ \sigma_{u^{k(1)}}\ast\cdots\ast\sigma_{u^{k(l)}} \ \perp \ \sigma_{u^{k'(1)}}\ast\cdots\ast\sigma_{u^{k'(l')}}\}.
\]
Since $G$ has the Baire property%(orthogonality of measures can be represented as intersection of conditions $C_k$ where $k$ is a natural number and $C_k$ states that there is a )
, by Theorem \ref{KU-var}, G is a comeager subset of $\mathcal{U}(H)$.
\bigskip\noindent
\end{proof}

\textit{Acknowledgement.} The present work is part of the author's PhD thesis under the supervision of Prof. Solecki. I am thankful to Prof. Solecki for many fruitful discussions and many on point suggestions and revisions. In particular, the idea of connecting the Glasner--Weiss problem with results in \cite{DL} and \cite{S1} is due to him.


\begin{thebibliography}{9}

%\bibitem{Chacon} \label{Cha_B}
%R. V. Chacon and T. Schwartzbauer,
%\textit{Commuting point transformations},
%Z. Wahrsch. Verw. Gebiete 11 (1969), 277–287.

\bibitem{Choksi} \label{Cho_B}
J. R. Choksi and M. G. Nadkarni,
\textit{Baire category in spaces of measure, unitary operators and transformations},
Inv. subspaces and allied topics, Narosa (1990), 147-163.
 
\bibitem{Sinai} \label{Sinai_B}
I. P. Cornfeld, S. V. Fomin, and Ya. G. Sinai,
\textit{Ergodic Theory},
Springer-Verlag, New York, 1982.

\bibitem{DL} \label{DL_B}
A. Del Junco and M. Lema\'nczyk,  
\textit{Generic spectral properties of measure-preserving maps and applications},
Proc. Amer. Math. Soc., Volume 115, Number 3, 1992.

%\bibitem{DL} \label{Fan_B}
%K. Fan,  
%\textit{On Positive Definite Sequences},
%Ann. of Math. 47(3), 1946. 

\bibitem{GW} \label{GW_B}
E. Glasner and B. Weiss,
\textit{Spatial and non-spatial actions of Polish groups},
Ergodic Theory Dynam. Systems 24 (2005), 1521--1538.

\bibitem{K} \label{Kat_B}
A. B. Katok, 
\textit{Combinatorial constructions in ergodic theory and dynamics}, 
Univ. Lecture Ser. 30, Amer. Math. Soc., Providence.

\bibitem{K1} \label{K1_B}
A. S. Kechris,
\textit{Classical Descriptive Set Theory},
Grad. Texts Math., Vol. 156, Springer-Verlag, 1995.
 
\bibitem{K2} \label{K2_B}
A. S. Kechris,
\textit{Global aspects of ergodic group actions},
Math. Surveys Monogr., Vol. 160, 2010.

\bibitem{King} \label{King_B}
J. L. King,
\textit{The generic transformation has roots of all orders},
Colloq. Math. 84/85(2000), 521–547.

%\bibitem{King} \label{Kin2_B}
%J. L. King,
%\textit{The commutant is the weak closure of the powers for rank-1 transformations},
%Ergodic Theory Dynam. Systems 6 (1986), 363–384.

\bibitem{MT} \label{MT_B}
J. Melleray and T. Tsankov,
\textit{Generic representations of abelian groups and extreme amenability}, 
Israel J. Math. 198(2013), 129--167.

\bibitem{N} \label{Nad_B}
M. G. Nadkarni,
\textit{Spectral theory of dynamical systems}, 
Birkhäuser Adv. Texts, Basel, 1998.

\bibitem{S2} \label{S2_B}
S. Solecki,
\textit{Closed subgroups generated by generic measure automorphisms},
Ergodic Theory Dynam. Systems 34(2014), 1011--1017.

\bibitem{S1} \label{S1_B}
S. Solecki,
\textit{Unitary representations of the groups of measurable and continuous functions with values in the circle},
J. Funct. Anal. 267(2014), 3105--3124.

\bibitem{S} \label{Ste_B}
A. M. Stepin, 
\textit{Spectral properties of generic dynamical systems}, 
Math. U.S.S.R. Izv. 29, 159-192.

\end{thebibliography}
\end{document}